\documentclass[12pt,a4paper,oneside]{amsart}
\usepackage[left=1.1in, right=0.8in, top=1.0in, bottom=1.0in]{geometry}
\usepackage{graphicx}
\usepackage{setspace}
\usepackage{indentfirst} 
\usepackage{cases}
\usepackage{wrapfig}
\usepackage[toc,page]{appendix}
\usepackage{amsmath}
\usepackage{amsthm}
\usepackage{amssymb}
\usepackage{enumerate}
\usepackage{mathrsfs}
\usepackage{dcolumn}
\usepackage{xcolor}
\usepackage{bookmark}

\newcolumntype{L}{D{.}{.}{2,5}}
\theoremstyle{plain}
\newtheorem{thm}{Theorem}[section]

\newtheorem{lemma}{Lemma}[section]
\newtheorem{remark}{Remark}[section]

\newtheorem{example}{Example}[section]
\newtheorem{corollary}{Corollary}[section]
\newtheorem{proposition}{Proposition}[section]
\newtheorem{question}{Question}[section]
\theoremstyle{definition}
 \usepackage{hyperref}
 \hypersetup{
    colorlinks=true,
    linkcolor=black,
    filecolor=magenta,      
    urlcolor=cyan}

\begin{document}
	\title[On N\"orlund summability of Taylor series]{On N\"orlund summability of Taylor series in\\ weighted Dirichlet spaces}
	\author{{A. B\"erd\"ellima}$^*$ and {N. L. Braha}$^{\dagger}$}
	\thanks{$^*$ berdellima@gmail.com, $\dagger$ nbraha@gmail.com}

	\AtEndDocument{\bigskip{%
	  \textsc[$^*$]{Berlin 10709, Germany.} \par
	  \addvspace{\medskipamount}
	  \textsc[$\dagger$]{University of Prishtina,\\ Faculty of Natural Science and Mathematics, \\Department of Mathematics and Computer Sciences,\\ Av. Mother Teresa, Nr.5, 10000 Prishtina, Kosovo.} \par	}}

%
	\maketitle
	\begin{abstract}In this note we show that the Taylor series of a function in a weighted Dirichlet space is (generalized) N\"orlund summable, provided that the sequence determining the N\"orlund operator is non-decreasing and has finite upper growth rate. In particular the Taylor series is N\"orlund summable for all $\alpha>1/2$, and the rate of convergence is of the order $O(n^{-1/2})$. The inequality $\alpha>1/2$ is sharp. On the other hand if the Taylor series is N\"orlund summable and the partial sums of the determining sequence enjoy a certain growth condition then the determining sequence has finite lower growth rate. An analogue result is derived for a non-increasing sequence that is uniformly bounded away from zero.
	\end{abstract}

\section{Introduction}
	A topic of interest in complex analysis is the approximation of a holomorphic function by its Taylor series. It is known that for certain spaces of holomorphic functions defined on the unit disk $\mathbb D$, the Taylor series of a function $f$ converges in norm of the space to $f$ itself. Such results hold in particular for the (general) Hardy space $H^p$ for any $1<p<+\infty$, or the Dirichlet space $\mathcal D$, the space of all holomorphic functions $f$ defined on $\mathbb D$ for which the Dirichlet integral 
	\begin{equation}
	\label{eq:Dirichlet-integral}
	\mathcal D(f)=\int_{\mathbb D}|f'(z)|^2\,dA(z)
	\end{equation}
	is finite, where $dA$ is the normalized area measure on $\mathbb D$. On the other hand there are spaces of holomorphic functions where approximation in norm by Taylor series fails. Such examples include for instance the Hardy space $H^1$ or the disk algebra $A(\mathbb D)$ as witnessed by the classical example of du Bois-Raymond \cite{du-Bois}. However by a result of \cite[Fej\'er]{Fejer, Fejer2} if instead the Ces\`aro sum is considered then Ces\`aro sums of Taylor series of a function converge in norm to the function itself in both aforementioned spaces. Later \cite[Riesz]{Riesz} refined the result of Fej\'er for the generalized Ces\'aro sums for functions in $A(\mathbb D)$. Recently as an application of their work about Hadamard multipliers on weighted Dirichlet spaces \cite[Mashreghi et al.]{Mashregi-Ransford}, Mashreghi--Ransford consider generalized Ces\`aro summability of Taylor series for functions in a weighted Dirichlet space \cite{Mashregi-Ransford2}. A weighted Dirichlet space $\mathcal D_{\omega}$ is similar to the Dirichlet space $\mathcal D$, except for \eqref{eq:Dirichlet-integral} is replaced by 
		\begin{equation}
		\label{eq:Dirichlet-weighted}
		\mathcal D_{\omega}(f)=\int_{\mathbb D}|f'(z)|^2\,\omega(z)\,dA(z),
		\end{equation}
		where $\omega$ is a superharmonic function on $\mathbb D$. These spaces were studied by \cite[Aleman]{Aleman}, which are a generalization of the case where $\omega$ is a harmonic function. The latter were originally introduced by \cite[Richter]{Richter} and later studied by \cite[Richter--Sundberg]{Richter2}.
	 In view of these recent developments we consider a different summability method, that of (generalized) N\"orlund sums in $\mathcal D_{\omega}$.
	We show convergence results for the N\"orlund method and we also provide rates of convergence. Differently from Ces\`aro method, the growth rate of the sequence that determines N\"orlund's operator plays a pivotal role. In particular we prove that when this sequence is non-decreasing and enjoys finite upper growth rate, then the Taylor series of a holomorphic function in $\mathcal D_{\omega}$ is N\"orlund summable in the norm of $\mathcal D_{\omega}$ (Theorem \ref{th:convergence}(i)) and that this rate of convergence is of the order $O(n^{-1/2})$ (Theorem \ref{th:quantitative}). Moreover under a certain growth condition for partial sums of the determining sequence of N\"orlund operator, we get that this sequence has finite lower growth rate, whenever Taylor series are N\"orlund summable (Theorem \ref{th:convergence}(ii)). Like in the case of generalized Ces\`aro method, the lower bound $\alpha>1/2$, the parameter in the generalized N\"orlund method, is sharp (Theorem \ref{th:convergence}(ii)). An analogue result (Theorem \ref{th:non-increasing}) is derived for non-increasing sequences that are uniformly bounded away from zero.

\section{Preliminary results}
\subsection{N\"orlund sums}
Let $f(z)$ be a holomorphic function of the complex variable $z$ and denote by $f(z)=\displaystyle\sum_{k=0}^{\infty}a_kz^k$ its formal power series expansion. Let
	\begin{equation}
	\label{eq:taylor}
	s_k[f](z)=a_0+a_1z+a_2z^2+\cdots+a_kz^k,\quad k\in \mathbb N_0
	\end{equation} 
be the $k$-th degree Taylor polynomial of $f$. Given a sequence of non-negative numbers $(p_n)_{n\in\mathbb N_0}$ such that $P_n=p_0+p_1+\cdots+p_n>0$ for all $n\in \mathbb N_0$, the N\"orlund operator $(N,(p_n)_{n\in\mathbb N})$ acts on a sequence $x=(x_n)_{n\in\mathbb N_0}$ by the formula
\begin{equation}
\label{eq:norlund}
(Nx)_n:=\frac{1}{P_n}\sum_{k=0}^np_{n-k}x_k,\quad n\in\mathbb N_0.
\end{equation}
For more on N\"orlund method and other summability methods we refer to \cite[Boos]{Boos}.
We define N\"orlund sums of a holomorphic function $f$ as follows
\begin{equation}
\label{eq:norlund-sums}
N_n[f](z):=\frac{1}{P_n}\sum_{k=0}^np_{n-k}s_k[f](z),\quad n\in\mathbb N_0.
\end{equation}

\begin{lemma}
\label{l:equiv}
The following holds 
\begin{equation}
\label{eq:equiv}
N_n[f](z)=\frac{1}{P_n}\sum_{k=0}^nP_{n-k}a_kz^k,\quad n\in\mathbb N_0.
\end{equation}
\end{lemma}
\begin{proof}
From \eqref{eq:norlund-sums} we have
\begin{align*}
N_n[f](z)=\frac{1}{P_n}\sum_{k=0}^n&p_{n-k}s_k[f](z)=\frac{1}{P_n}\sum_{k=0}^np_{n-k}\sum_{j=0}^ka_jz^j\\&
=\frac{1}{P_n}((P_na_0)+(P_{n-1}a_1z)+\cdots+(P_0a_nz^n))
=\frac{1}{P_n}\sum_{k=0}^nP_{n-k}a_kz^k.
\end{align*}
\end{proof}
In view of Lemma \ref{l:equiv} we define 
\begin{equation}
\label{eq:generalized-norlund}
N^{\alpha}_n[f](z):=\frac{1}{P_n^{\alpha}}\sum_{k=0}^nP^{\alpha}_{n-k}a_kz^k,\quad f(z)=\sum_{k=0}^{\infty}a_kz^k
\end{equation}
to be the generalized N\"orlund sum $N_n^{\alpha}[f]$ with parameter $\alpha>0$ of a function $f$. It is evident that $N^{\alpha}_n$ is a linear operator acting on the space of all holomorphic functions and its image is a certain subset of polynomials of degree $\leq n$ defined on the complex plane $\mathbb C$.

\subsection{Superharmonic functions and weighted Dirichlet spaces}
Let $U\subseteq\mathbb C$ be an open set. A function $\omega:U\to[-\infty,+\infty)$ is subharmonic on $U$ if it is upper semicontinuous and given $z\in U$ there exists $R>0$ such that 
\begin{equation}
\label{eq: subharmonic}
\omega(z)\leq \frac{1}{2\pi}\int^{2\pi}_0\omega(z+re^{i\theta})\,d\theta\quad(0\leq r<R).
\end{equation}
A function $\omega:U\to(-\infty,+\infty]$ is superharmonic on $U$ if $-\omega$ is subharmonic on $U$. In particular a harmonic function is both sub- and superharmonic, see e.g. \cite[Ransford \S 2.2]{Ransford}. Denote by $\mathbb D=\{z\in\mathbb C\,:\,|z|< 1\}$ the unit disk in $\mathbb C$, let $\omega>0$ be superharmonic, and $f$ holomorphic on $\mathbb D$. Define
\begin{equation}
\label{eq:dirichlet}
\mathcal D_{\omega}(f):=\int_{\mathbb D}|f'(z)|^2\omega(z)\,dA(z),
\end{equation}
where $dA$ denotes the normalized area measure on $\mathbb D$. The weighted Dirichlet space $\mathcal D_{\omega}$ is the set of all holomorphic functions $f$ on $\mathbb D$ such that $\mathcal D_{\omega}(f)<+\infty$. 
When equipped with the inner product
\begin{equation}
\label{eq:inner-product}
\langle f,g\rangle_{\mathcal D_{\omega}}:=f(0)\,\overline{g(0)}+\int_{\mathbb D}f'(z)\overline{g'(z)}\,\omega(z)\,dA(z),
\end{equation} 
$\mathcal D_{\omega}$ becomes a Hilbert space. This induces the norm
\begin{equation}
\label{eq:dirichlet-norm}
\|f\|^2_{\mathcal D_{\omega}}:=|f(0)|^2+\mathcal D_{\omega}(f),\quad f\in\mathcal D_{\omega}.
\end{equation} 
If $\omega$ is a positive superharmonic function, then there exists a unique positive finite Borel measure $\mu_{\omega}$ on $\overline{\mathbb D}$, such that for all $z\in\mathbb D$ it holds
\begin{equation}
\label{eq:Borel-measure}
\omega(z)=\int_{\mathbb D}\log\Big|\frac{1-\overline{\zeta}\,z}{\zeta-z}\Big|\,\frac{2}{1-|\zeta|^2}\,d\mu_{\omega}(\zeta)+\int_{\partial{\mathbb D}}\frac{1-|z|^2}{|\zeta-z|^2}\,d\mu_{\omega}(\zeta).
\end{equation}
When $\mu_{\omega}=\delta_{\zeta}$ (the Dirac mass at $\zeta$), we write $\mathcal D_{\zeta}$ for $\mathcal D_{\omega}$, see e.g. \cite[Theorem 4.5.1]{Ransford}. Note that one of the integrals above vanishes depending on $\zeta\in\mathbb D$ or $\zeta\in\partial\mathbb D$. From the properties of Dirac delta function $\delta_{\zeta}$ for a given holomorphic function $f$ we then have 
\begin{equation}
\label{eq:slice}
\mathcal D_{\zeta}(f)=\left\{\begin{array}{ccc}
\displaystyle\int_{\mathbb D}\log\Big|\frac{1-\overline{\zeta}\,z}{\zeta-z}\Big|\,\frac{2}{1-|\zeta|^2}\,|f'(z)|^2d A(z),&\zeta\in\mathbb D\\[1em]
\displaystyle\int_{\mathbb D}\frac{1-|z|^2}{|\zeta-z|^2}\,|f'(z)|^2\,dA(z),& \zeta\in\partial\mathbb D.
\end{array}\right.
\end{equation}
By virtue of Fubini's Theorem one can recover $\mathcal D_{\omega}(f)$ from $\mathcal D_{\zeta}(f)$ from the integral formula $\mathcal D_{\omega}(f)=\displaystyle\int_{\overline{\mathbb D}}\mathcal D_{\zeta}(f)\,d\mu_{\omega}(\zeta)$. Often $\mathcal D_{\zeta}$ is called a local Dirichlet space. If $H^2$ denotes the Hardy space, the space of holomorphic functions $f$ on $\mathbb D$ such that $\|f\|_{H^2}^2=\displaystyle\sum_{k=0}^{\infty}|a_k|^2<\infty$,
then we have the characterization due to \cite[Aleman \S 4.1]{Aleman}, see e.g. \cite[Theorem 2.1]{Mashregi-Ransford}:
\begin{lemma}
\label{l:Aleman} Let $\zeta\in\overline{\mathbb D}$, then $f\in\mathcal D_{\zeta}$ if and only if there exists $g\in H^2$ and $a\in\mathbb C$ such that $f(z)=a+(z-\zeta)\,g(z)$ for all $z\in\mathbb D$, and in this case $D_{\zeta}(f)=\|g\|_{H^2}$.
\end{lemma}
Given a holomorphic function $f\in\mathcal D_{\omega}$ we say that the Taylor series of $f$ is (N\"orlund) $(N,\alpha)$--summable, if $\|N^{\alpha}_n[f]-f\|_{\mathcal D_{\omega}}\to 0$ as $n\to+\infty$.

\subsection{Hadamard product}

Let $f(z)=\displaystyle\sum_{k=0}^{\infty}a_kz^k$ and $g(z)=\displaystyle\sum_{k=0}^{\infty}b_kz^k$ be two formal power series. The Hadamard product of $f$ and $g$ is defined as 
\begin{equation}
\label{eq:hadamard}
(f\ast g)(z):=\sum_{k=0}^{+\infty}a_kb_kz^k.
\end{equation}
It is clear that Hadamard product is commutative and associative operation. Moreover if $f$ or $g$ is a polynomial then $(f\ast g)$ is a polynomial too, and if both $f$ and $g$ are holomorphic then so is $f\ast g$. If $h$ is a power series with the property that $h\ast f\in\mathcal D_{\omega}$ whenever $f\in\mathcal D_{\omega}$, then $h$ is a Hadamard multiplier. 

\begin{proposition}
\label{p:hadamard-multiplier}
$h$ is a Hadamard multiplier for any positive superharmonic weight $\omega$, whenever $h$ is a polynomial. 
\end{proposition}
\begin{proof}
Let $h=b_0+b_1z+\cdots+b_nz^n$ be a polynomial then so is $h\ast f$ and in particular $h\ast f$ is holomorphic on $\mathbb D$. It only remains to show that $\mathcal D_{\omega}(h\ast f)<+\infty$. Note that by definition of $\mathcal D_{\omega}$ we have
\begin{align*}
\mathcal D_{\omega}(h\ast f)&=\int_{\mathbb D}|(h\ast f)'(z)|^2\,\omega(z)\,dA(z)
=\int_{\mathbb D}|\sum_{k=1}^na_kb_kz^k|^2\,\omega(z)\,dA(z)\\&\leq \int_{\mathbb D}(\sum_{k=1}^n|a_k||b_k||z|^k)^2\omega(z)\,dA(z)\leq C_n\int_{\mathbb D}\omega(z)\,dA(z)
\end{align*}
where $C_n:=(\displaystyle\sum_{k=1}^n|a_k||b_k|)^2$ is a certain positive constant depending on $n\in\mathbb N_0$. From the general theory of harmonic functions, e.g. see \cite[Theorem 2.5.1 \S 2.5]{Ransford}, if $\omega$ is positive and superharmonic on $\mathbb D$ then $\omega$ is integrable on $\mathbb D$. Consequently $\mathcal D_{\omega}(h\ast f)<+\infty$.
\end{proof}
 Given $n\in\mathbb N$ define
\begin{equation}
\label{eq:hn}
h^{\alpha}_n(z)=\frac{1}{P_n^{\alpha}}\sum_{k=0}^{n}P_{n-k}^{\alpha}z^k.
\end{equation} 
In view of Hadamard product we can then write $N^{\alpha}_n[f](z)=(h^{\alpha}_n\ast f)(z)$ for every $n\in\mathbb N_0$.
\begin{remark}
Since $h^{\alpha}_n$ in \eqref{eq:hn} is a polynomial for every $n\in\mathbb N_0$ then $h^{\alpha}_n$ is a Hadamard multiplier for any positive superharmonic weight $\omega$.
\end{remark}

Define the matrix $T^{\alpha}_{n}$ as 
\begin{equation}
\label{eq:matrix-T}
T^{\alpha}_{n}:=\begin{pmatrix}
c_1&(c_2-c_1)& (c_3-c_2)&\cdots \\
0&c_2&(c_3-c_2)&\cdots\\
0&0 & c_3&\cdots\\
\cdots&\cdots&\cdots&\cdots\\
\end{pmatrix},\quad \text{where}\;c_k:= \left\{
\begin{array}{ll}
     \displaystyle\Big(\frac{P_{n-k}}{P_n}\Big)^{\alpha} &  1\leq k\leq n\\[1em]
     0 & k>n.
\end{array} 
\right.
\end{equation}

Since $h^{\alpha}_n$ is a polynomial for each $n\in\mathbb N_0$, then by \cite[Theorem 2.1]{Mashregi-Ransford2} we have:
\begin{corollary}
\label{c:N-T}
For every superharmonic weight $\omega$ it holds that $\|N^{\alpha}_n\|_{\mathcal D_{\omega}\to\mathcal D_{\omega}}\leq \|T^{\alpha}_{n}\|_{\ell^2\to\ell^2}$, with equality if $\omega(z)=(1-|z|^2)/|1-z|^2$.
\end{corollary}
As a consequence of \cite[Theorem 2.2 (i)]{Mashregi-Ransford2} we have the lemma:
\begin{lemma}
\label{l:bound-T}
For each $n\in\mathbb N$ it holds that 
\begin{equation}
\label{eq:bound-T}
\|T_n^{\alpha}\|^2_{\ell^2\to\ell^2}\leq (n+1)\sum_{k=1}^n\Big|\Big(\frac{P_{n-k}}{P_n}\Big)^{\alpha}-\Big(\frac{P_{n-k-1}}{P_n}\Big)^{\alpha}\Big|^2.
\end{equation}
\end{lemma}
\subsection{Growth rate}
The upper growth rate $\rho$ of the sequence $(p_n)_{n\in\mathbb N_0}$ is defined as 
\begin{equation}
\label{eq:lower}
\rho:=\limsup_{n\to+\infty}\rho_n,\quad\rho_n:=\frac{n\,p_n}{P_n}.
\end{equation}
If $\rho<+\infty$ we say that the sequence $(p_n)_{n\in\mathbb N_0}$ has finite upper growth rate. Similarly one can define the lower growth rate if limsup is replaced by liminf. Evidently a sequence with finite upper growth rate has finite lower growth rate. For instance the sequence $p_n=n^k$ for any $k\in\mathbb N$ has finite growth rate. An example of a sequence with unbounded growth rate is $p_n=r^n$ for any $r>1$ or the sequence $p_n=\ln(n+1)$ for $n\in\mathbb N$. Note that the quantity $\rho_n$ appears in \cite{Borwein} and plays a crucial role in determining the $\ell_p$ norm of the N\"orlund operator. 
The following lemma is necessary for our main results; it ties together the importance of growth rate of the determining sequence $(p_n)_{n\in\mathbb N_0}$ with convergence of N\"orlund method.
\begin{lemma}
\label{l:polynomials} Let $(p_n)_{n\in\mathbb N_0}$ have finite upper growth rate, then for any polynomial $f$ it holds that $\|N^{\alpha}_n[f]-f\|_{\mathcal D_{\omega}}\to 0$ as $n\to+\infty$. 
\end{lemma}
\begin{proof}
Let $f(z)=a_0+a_1z+\cdots +a_mz^m$ for some $m\in\mathbb N_0$, then 
$$N^{\alpha}_n[f](z)=\frac{1}{P_n^{\alpha}}\sum_{k=0}^m P_{n-k}^{\alpha}a_kz^k$$
implies
\begin{align*}
\|N^{\alpha}_n[f]-f\|_{\mathcal D_{\omega}}&=\int_{\mathbb D}|N_n^{\alpha}[f]'(z)-f'(z)|^2\,\omega(z)\,dA(z)\\&=\int_{\mathbb D}|\sum_{k=1}^m\Big(\Big(\frac{P_{n-k}}{P_n}\Big)^{\alpha}-1\Big)ka_kz^{k-1}|^2\,\omega(z)\,dA(z)\leq C^{\alpha}_{m,n}\int_{\mathbb D}\omega(z)\,dA(z)
\end{align*}
where 
\begin{equation}
\label{eq:Cmn}
C^{\alpha}_{m,n}=\Big(\sum_{k=1}^m\Big(\Big(\frac{P_{n-k}}{P_n}\Big)^{\alpha}-1\Big)k|a_k|\Big)^2.
\end{equation}
We claim that $\lim_{n\to+\infty}P_{n-k}/P_n=1$ for every $k\in\{1,2,\cdots,m\}$. We proceed by strong induction. For $k=1$ we have
$$\limsup_{n\to+\infty}\frac{P_{n-1}}{P_n}=1-\liminf_{n\to+\infty}\frac{p_n}{P_n}=1-\lim_{n\to+\infty}\frac{1}{n}\cdot\liminf_{n\to+\infty}\frac{np_n}{P_n}=1-0\cdot \liminf_{n\to+\infty}\frac{np_n}{P_n}=1.$$
Note that finite upper growth rate implies finite lower growth rate and therefore the last equality is justified. Similarly 
$$\liminf_{n\to+\infty}\frac{P_{n-1}}{P_n}=1-\limsup_{n\to+\infty}\frac{p_n}{P_n}=1-\lim_{n\to+\infty}\frac{1}{n}\cdot\limsup_{n\to+\infty}\frac{np_n}{P_n}=1-0\cdot \limsup_{n\to+\infty}\frac{np_n}{P_n}=1.$$
 Consequently $\lim_{n\to+\infty}P_{n-1}/P_n=1$. Now let the claim holds for all $k\in\{2,\cdots,m-1\}$ and set $k=m$. Then we obtain 
\begin{align*}
\limsup_{n\to+\infty}\frac{P_{n-m}}{P_n}&=\limsup_{n\to+\infty}\frac{P_{n-m}}{P_{n-m+1}}\frac{P_{n-m+1}}{P_n}=\limsup_{n\to+\infty}\frac{P_{n-m}}{P_{n-m+1}}\cdot\lim_{n\to+\infty}\frac{P_{n-m+1}}{P_n}=\limsup_{n\to+\infty}\frac{P_{n-m}}{P_{n-m+1}}\\&
=1-\liminf_{n\to+\infty}\frac{p_{n-m+1}}{P_{n-m+1}}=1-\lim_{n\to+\infty}\frac{1}{n-m+1}\cdot\liminf_{n\to+\infty}\frac{(n-m+1)p_{n-m+1}}{P_{n-m+1}}\\&
=1-0\cdot \liminf_{n\to+\infty}\frac{(n-m+1)p_{n-m+1}}{P_{n-m+1}}=1.
\end{align*}
Similarly one computes $\liminf_{n\to+\infty}P_{n-m}/P_n=1$, hence $\lim_{n\to+\infty}P_{n-m}/P_n=1$.
Then it follows 
\begin{align*}
 \lim_{n\to+\infty}C^{\alpha}_{m,n}=\lim_{n\to+\infty}\Big(\sum_{k=1}^m\Big(\Big(\frac{P_{n-k}}{P_n}\Big)^{\alpha}-1\Big)k|a_k|\Big)^2= \Big(\sum_{k=1}^m\Big(\lim_{n\to+\infty}\Big(\frac{P_{n-k}}{P_n}\Big)^{\alpha}-1\Big)k|a_k|\Big)^2= 0.
\end{align*}
 Consequently $\lim_{n\to+\infty}\|N^{\alpha}_n[f]-f\|_{\mathcal D_{\omega}}=0$. This completes the proof.
\end{proof}

\section{N\"orlund summability of Taylor series}

\begin{thm}
\label{th:convergence} Let $(N, (p_n)_{n\in\mathbb N_0})$ be the N\"orlund operator such that $(p_n)_{n\in\mathbb N_0}$ is non-decreasing and let $\alpha>1/2$. Then the followings are true:
\begin{enumerate}[(i)]
\item If $(p_n)_{n\in\mathbb N_0}$ has finite upper growth, then $\lim_{n\to+\infty}\|N^{\alpha}_n[f]-f\|_{\mathcal D_{\omega}}=0$ for every $f\in\mathcal D_{\omega}$.
\item If additionally the growth condition 
\begin{equation}
\label{eq:growth-Pn}
\liminf_{n\to+\infty}\frac{P_{n-1}}{P_{2n}}>0
\end{equation}
is satisfied, then $\lim_{n\to+\infty}\|N^{\alpha}_n[f]-f\|_{\mathcal D_{\omega}}=0$ for every $f\in\mathcal D_{\omega}$ implies that $(p_n)_{n\in\mathbb N_0}$ has finite lower growth rate. Moreover the inequality $\alpha>1/2$ is sharp.
\end{enumerate}
\end{thm}
\begin{proof}
\begin{enumerate}[(i)]
\item Let $\rho=\limsup_{n\to+\infty}\rho_n<+\infty$. By Lemma \ref{l:bound-T} we have 
\begin{align*}
\|T^{\alpha}_{n}\|^2_{\ell^2\to\ell^2}\leq (n+1)\sum_{k=1}^n\Big|\Big(\frac{P_{n-k}}{P_n}\Big)^{\alpha}-\Big(\frac{P_{n-k-1}}{P_n}\Big)^{\alpha}\Big|^2=\frac{n+1}{P_n^{2\alpha}}\sum_{k=1}^n|P_{n-k}^{\alpha}-P_{n-k-1}^{\alpha}|^2.
\end{align*}
By Mean Value Theorem, e.g. see \cite[Theorem 8, pp.178-9]{Spivak} for each $k=1,2,\cdots, n$ there is $\xi_k\in(P_{n-k-1},P_{n-k})$ such that
$|P_{n-k}^{\alpha}-P_{n-k-1}^{\alpha}|=|P_{n-k}-P_{n-k-1}|\,\alpha \xi^{\alpha-1}_{k}$ and consequently we get the upper estimate 
\begin{align*}
\|T^{\alpha}_{n}\|^2_{\ell^2\to\ell^2}&\leq\alpha^2\frac{n+1}{P_n^{2\alpha}}\sum_{k=1}^n|P_{n-k}-P_{n-k-1}|^2P^{2\alpha-2}_{n-k}\\&\leq \alpha^2\frac{n+1}{P_n^{2\alpha}}\sum_{k=1}^n(n-k+1)^{2\alpha-2}\,p^{2\alpha}_{n-k}\\&= \alpha^2\frac{n+1}{P_n^{2\alpha}}\sum_{k=1}^nk^{2\alpha-2}\,p^{2\alpha}_{n-k}\leq \left\{
\begin{array}{ll}
     \displaystyle\alpha^2\frac{n+1}{n}\rho_n^{2\alpha} &  \alpha\geq 1\\[1em]
    \displaystyle\frac{\alpha^2}{2\alpha-1}\frac{(n+1)(n^{2\alpha-1}+2\alpha-2)}{n^{2\alpha}}\rho_n^{2\alpha} & 1/2<\alpha<1.
     \end{array}\right.
\end{align*}
For every $\varepsilon>0$ there exists $n(\varepsilon)\in\mathbb N_0$ such that 
\begin{align*}
\|T^{\alpha}_{n}\|^2_{\ell^2\to\ell^2}< \left\{
\begin{array}{ll}
     \displaystyle\alpha^2\,\rho^{2\alpha}+\varepsilon &  \alpha\geq 1\\[1em]
    \displaystyle\frac{\alpha^2}{2\alpha-1}\,\rho^{2\alpha}+\varepsilon& 1/2<\alpha<1.
     \end{array}\right.
\quad\text{for all}\;n\geq n(\varepsilon).
\end{align*}
Now let $f\in\mathcal D_{\omega}$, since polynomials are dense in $\mathcal D_{\omega}$ then there exists a sequence of polynomials $(f_k)_{k\in\mathbb N}\subset\mathcal D_{\omega}$ such that $\|f_k-f\|_{\mathcal D_{\omega}}\to 0$ as $k\to+\infty$. Moreover
\begin{align*}
\|N^{\alpha}_n[f]-f\|_{\mathcal D_{\omega}}&\leq \|N^{\alpha}_n[f]-N^{\alpha}_n[f_k]\|_{\mathcal D_{\omega}}+\|N^{\alpha}_n[f_k]-f_k\|_{\mathcal D_{\omega}}+\|f_k-f\|_{\mathcal D_{\omega}}\\&
\leq (\|N^{\alpha}_n\|_{\mathcal D_{\omega}\to\mathcal D_{\omega}}+1)\|f_k-f\|_{\mathcal D_{\omega}}+\|N^{\alpha}_n[f_k]-f_k\|_{\mathcal D_{\omega}}.
\end{align*}
By Lemma \ref{l:polynomials} $\|N^{\alpha}_n[f_k]-f_k\|_{\mathcal D_{\omega}}\to 0$ as $n\to+\infty$. By Corollary \ref{c:N-T} we have that $$\limsup_{n\to+\infty}\|N^{\alpha}_n\|_{\mathcal D_{\omega}\to\mathcal D_{\omega}}\leq \left\{
\begin{array}{ll}
     \displaystyle\alpha\,\rho^{\alpha} &  \alpha\geq 1\\[1em]
    \displaystyle\frac{\alpha}{(2\alpha-1)^{1/2}}\,\rho^{\alpha}& 1/2<\alpha<1,
     \end{array}\right.,$$ consequently by letting $k\to+\infty$ the first term in the last inequality vanishes. 
     \item Now suppose that $(p_n)_{n\in\mathbb N_0}$ has unbounded lower growth rate. Let 
     $$\omega(z)=\frac{1-|z|^2}{|1-z|^2},\quad z\in\mathbb D,$$
     then by Corollary \ref{c:N-T} we have $\|N^{\alpha}_n\|_{\mathcal D_{\omega}\to\mathcal D_{\omega}}=\|T^{\alpha}_{n}\|_{\ell^2\to\ell^2}$.
     By \cite[Theorem 2.2 (ii)]{Mashregi-Ransford2}, but instead applying it to the $(n-m+1)\times m$ submatrix
     $$A=\begin{pmatrix}
     (c_2-c_1)&(c_3-c_2)&\cdots&(c_{n-m+1}-c_{n-m})\\
     (c_2-c_1)&(c_3-c_2)&\cdots&(c_{n-m+1}-c_{n-m})\\
     \cdots&\cdots&\cdots&\cdots\\
     (c_2-c_1)&(c_3-c_2)&\cdots&(c_{n-m+1}-c_{n-m})
     \end{pmatrix},\; \text{where}\;c_k:= \left\{
     \begin{array}{ll}
          \displaystyle\Big(\frac{P_{n-k}}{P_n}\Big)^{\alpha} &  1\leq k\leq n\\[1em]
          0 & k>n
          \end{array}\right.$$
     we get for any $1\leq m\leq n$ 
     \begin{align*}
     \|T^{\alpha}_{n}\|_{\ell^2\to\ell^2}^2\geq \|A\|^2&=\|AA^*\|=\frac{n-m+1}{P_n^{2\alpha}}\sum_{k=1}^m|P_{n-k}^{\alpha}-P_{n-k-1}^{\alpha}|^2\\&=\alpha^2\frac{n-m+1}{P_n^{2\alpha}}\sum_{k=1}^m|P_{n-k}-P_{n-k-1}|^2\,\xi_k^{2\alpha-2},
     \end{align*}
     where $\xi_k\in(P_{n-k-1},P_{n-k})$. From monotonicity of the sequence $(p_n)_{n\in\mathbb N_0}$ and the recurrence relation $P_{n-k}=P_{n-k-1}+p_{n-k}$ we obtain the lower bound
     \begin{align*}
     \|T^{\alpha}_{n}\|_{\ell^2\to\ell^2}^2&\geq\alpha^2\frac{n-m+1}{P_n^{2\alpha}}\sum_{k=1}^mp_{n-k}^2\,P_{n-k-1}^{2\alpha-2}\\&\geq \alpha^2\frac{n-m+1}{P_n^{2\alpha}}\sum_{k=1}^m(m-k+1)^{2\alpha-2}\,p_{n-m-1}^{2\alpha}= \alpha^2(n-m+1)\Big(\frac{p_{n-m-1}}{P_n}\Big)^{2\alpha}\sum_{k=1}^mk^{2\alpha-2}.
     \end{align*}
     For $\alpha>1/2$ the last finite sum can be estimated by the elementary integrals
     \begin{equation}
     \label{eq:estimate-sum}
     \sum_{k=1}^mk^{2\alpha-2}\geq \left\{
     \begin{array}{ll}
          \displaystyle\int^{m}_0t^{2\alpha-2}\,dt=\frac{m^{2\alpha-1}}{2\alpha-1} &  \alpha\geq 1\\[1em]
         \displaystyle\int^{m+1}_1t^{2\alpha-2}\,dt=\frac{(m+1)^{2\alpha-1}-1}{2\alpha-1}& 1/2<\alpha<1,
          \end{array}\right.
     \end{equation}
     implying finally the lower bound for the operator norm as follows
     $$\|T^{\alpha}_{n}\|_{\ell^2\to\ell^2}^2\geq
     \left\{
     \begin{array}{ll}
          \displaystyle\frac{\alpha^2}{2\alpha-1}\frac{n-m+1}{m}\Big(\frac{m\,p_{n-m-1}}{P_n}\Big)^{2\alpha} &  \alpha\geq 1\\[1em]
         \displaystyle\frac{\alpha^2}{2\alpha-1}\frac{(n-m+1)((m+1)^{2\alpha-1}-1)}{m^{2\alpha}}\Big(\frac{m\,p_{n-m-1}}{P_n}\Big)^{2\alpha} & 1/2<\alpha<1.
          \end{array}\right.$$
     Restricting to even indices $2n$ and taking $m=n$ yields
     \begin{align*}
     \|T^{\alpha}_{2n}\|_{\ell^2\to\ell^2}^2&\geq\left\{
     \begin{array}{ll}
          \displaystyle\frac{\alpha^2}{2\alpha-1}\frac{n+1}{n}\Big(\frac{n\,p_{n-1}}{P_{n-1}}\Big)^{2\alpha}\Big(\frac{P_{n-1}}{P_{2n}}\Big)^{2\alpha} &  \alpha\geq 1\\[1em]
         \displaystyle\frac{\alpha^2}{2\alpha-1}\frac{(n+1)((n+1)^{2\alpha-1}-1)}{n^{2\alpha}}\Big(\frac{n\,p_{n-1}}{P_{n-1}}\Big)^{2\alpha}\Big(\frac{P_{n-1}}{P_{2n}}\Big)^{2\alpha} & 1/2<\alpha<1.
          \end{array}\right.
     \end{align*}
     Passing in lower limit for any $\alpha>1/2$ yields
     \begin{align*}
     \liminf_{n\to+\infty}\|T^{\alpha}_{2n}\|_{\ell^2\to\ell^2}^2&\geq\frac{\alpha^2}{2\alpha-1}\liminf_{n\to+\infty}\Big[\Big(\frac{n\,p_{n-1}}{P_{n-1}}\Big)^{2\alpha}\Big(\frac{P_{n-1}}{P_{2n}}\Big)^{2\alpha}\Big]\\&
     \geq \frac{\alpha^2}{2\alpha-1}\liminf_{n\to+\infty}\Big(\frac{n\,p_{n-1}}{P_{n-1}}\Big)^{2\alpha}\cdot\liminf_{n\to+\infty}\Big(\frac{P_{n-1}}{P_{2n}}\Big)^{2\alpha}.
     \end{align*}
     By growth condition \eqref{eq:growth-Pn} we have that $\beta=\liminf_{n\to+\infty}P_{n-1}/P_{2n}>0$.
     Moreover $\beta\leq 1$, since $P_{n-1}\leq P_{2n}$. Then we obtain
     \begin{align*}
     \liminf_{n\to+\infty}\|T^{\alpha}_{2n}\|_{\ell^2\to\ell^2}^2&\geq
     \beta^{2\alpha}\,\frac{\alpha^2}{2\alpha-1}\liminf_{n\to+\infty}\Big(\frac{n\,p_{n-1}}{P_{n-1}}\Big)^{2\alpha},
     \end{align*}
     which tends to infinity if the sequence $(p_n)_{n\in\mathbb N_0}$ has unbounded lower growth rate for any $\alpha>1/2$.
     Consequently from the inequalities
      $$\limsup_{n\to+\infty}\|T^{\alpha}_{n}\|_{\ell^2\to\ell^2}\geq\limsup_{n\to+\infty}\|T^{\alpha}_{2n}\|_{\ell^2\to\ell^2}\geq\liminf_{n\to+\infty}\|T^{\alpha}_{2n}\|_{\ell^2\to\ell^2}$$
     and $\|N^{\alpha}_n\|_{\mathcal D_{\omega}\to\mathcal D_{\omega}}=\|T^{\alpha}_{n}\|_{\ell^2\to\ell^2}$ it follows that
      $\limsup_{n\to+\infty}\|N^{\alpha}_{n}\|_{\mathcal D_{\omega}\to\mathcal D_{\omega}}=+\infty$ and so $\sup_{n\in\mathbb N_0}\|N^{\alpha}_{n}\|_{\mathcal D_{\omega}\to\mathcal D_{\omega}}=+\infty$. By virtue of uniform boundedness principle (Banach--Steinhaus Theorem, e.g. see \cite[Theorem 1, p.68]{Yosida}), there exists $f\in\mathcal D_{\omega}$ satisfying $\sup_{n\in\mathbb N_0}\|N^{\alpha}_{n}[f]\|_{\mathcal D_{\omega}}=+\infty$, i.e. $\lim_{n\to+\infty}\|N^{\alpha}_{n}[f]-f\|_{\mathcal D_{\omega}}\neq 0$.
       Now let $\alpha=1/2$, then $$\sum_{k=1}^nk^{-1}\geq \int^{n+1}_1t^{-1}\,dt=\ln(n+1).$$
       On the other hand note that $P_n\leq (n+1)p_n$ implies $\rho_n\geq n/(n+1)$ for all $n\in\mathbb N$. 
      Thus we obtain 
      $$\liminf_{n\to+\infty}\|T^{1/2}_{2n}\|_{\ell^2\to\ell^2}^2\geq
           \beta\,\liminf_{n\to+\infty}(\ln(n+1)\rho_n)\geq\beta\,\liminf_{n\to+\infty}\ln(n+1)=+\infty.$$ Again by uniform boundedness principle the claim follows.
           
\end{enumerate}

\end{proof}

\begin{example}
Let $p_n=1$ for all $n\in\mathbb N_0$, then $P_n=n+1$ and $\lim_{n\to+\infty}\rho_n=1$. Moreover $p_n\leq p_{n+1}$, conditions of Theorem \ref{th:convergence} are fulfilled, thus $\lim_{n\to+\infty}\|N^{\alpha}_n[f]-f\|_{\mathcal D_{\omega}}=0$. 
\end{example}

\begin{example}
Let $p_n=n$ for $n\in\mathbb N$, then $P_n=n(n-1)/2$ and $\lim_{n\to+\infty}\rho_n=2$. Moreover $p_n< p_{n+1}$, conditions of Theorem \ref{th:convergence} are fulfilled, thus $\lim_{n\to+\infty}\|N^{\alpha}_n[f]-f\|_{\mathcal D_{\omega}}=0$.
\end{example}
Another aspect we would like to touch is the rate at which $N_n^{\alpha}[f]$ converges towards $f$ as $n\to+\infty$. 
The next result gives essentially a quantitative statement of Theorem \ref{th:convergence}. 

\begin{thm}
\label{th:quantitative}
Let $(N, (p_n)_{n\in\mathbb N_0})$ be the N\"orlund operator with $(p_n)_{n\in\mathbb N_0}$ non-decreasing and $\alpha>1/2$. If $(p_n)_{n\in\mathbb N_0}$ has finite upper growth, then $\|N^{\alpha}_n[f]-f\|_{\mathcal D_{\omega}}=O(n^{-1/2})$ for all $f\in\mathcal D_{\omega}$.
\end{thm}
\begin{proof} Let $\rho<+\infty$ be the upper growth rate of $(p_n)_{n\in\mathbb N_0}$. We consider only the case when $\alpha\geq 1$. The situation $1/2<\alpha<1$ is treated analogously by just replacing the corresponding operator norms from Theorem \ref{th:convergence}.
It suffices to prove that there is some positive constant $C$, possibly depending on $f, \omega,\rho$ and $\alpha$ such that $\|N^{\alpha}_n[f]-f\|_{\mathcal D_{\omega}}\leq C/n^{1/2}$ for all sufficiently large $n$. For every $k\in\mathbb N$ there is a polynomial $f_{k}$ with $\|f-f_{k}\|_{\mathcal D_{\omega}}< 1/k^5$. On the other hand we have the estimate
\begin{align}
\label{eq:triangle}
\|N^{\alpha}_n[f]-f\|_{\mathcal D_{\omega}}\leq \|N^{\alpha}_n[f]-N^{\alpha}_n[f_{k}]\|_{\mathcal D_{\omega}}+\|N^{\alpha}_n[f_{k}]-f_{k}\|_{\mathcal D_{\omega}}+\|f_{k}-f\|_{\mathcal D_{\omega}}.
\end{align}
For the first term we have $$\|N^{\alpha}_n[f]-N^{\alpha}_n[f_{k}]\|_{\mathcal D_{\omega}}\leq \|N^{\alpha}_n\|_{\mathcal D_{\omega}\to\mathcal D_{\omega}}\,\|f-f_{k}\|_{\mathcal D_{\omega}}\leq 2\alpha\rho^{\alpha}\,\frac{1}{k^5},\quad \text{for sufficiently large}\;n.$$
Next we get an upper estimate for the middle term in \eqref{eq:triangle}. Take $n\geq k$ then
\begin{align*}
\|N_n^{\alpha}[f_k]-f_k\|^2_{\mathcal D_{\omega}}&=\int_{\mathbb D}|N^{\alpha}_n[f_k]'(z)-f_k'(z)|^2\,\omega(z)\,dA(z)\\&=\int_{\mathbb D}\Big|\sum_{j=1}^k\Big(\frac{P_{n-j}}{P_n}\Big)^{\alpha}-1\Big)ja^{(k)}_jz^{j-1}\Big|^2\,\omega(z)\,dA(z)\\&
\leq \Big(\sum_{j=1}^k\Big|\Big(\frac{P_{n-j}}{P_n}\Big)^{\alpha}-1\Big|j|a^{(k)}_j|\Big)^2\,\int_{\mathbb D}\omega(z)\,dA(z)=C(n,\alpha, f_k)\cdot C(\omega).
\end{align*}
We need only then to show that the positive constant $C(n,\alpha, f_k)$ depends only on $\alpha, f$ for sufficiently large $k$ (and also $n$). Note that $C(\omega)>0$ since $\omega$ is a positive superharmonic function on $\mathbb D$. Next we find an upper estimate for the first constant above
\begin{align*}
C(n,\alpha,f_k)=\Big(\sum_{j=1}^k\Big|\Big(\frac{P_{n-j}}{P_n}\Big)^{\alpha}-1\Big|j|a_j^{(k)}|\Big)^2&\leq \Big(\sum_{j=1}^k\Big|\Big(\frac{P_{n-j}}{P_n}\Big)^{\alpha}-1\Big|^2\Big)\cdot \Big(\sum_{j=1}^kj^2|a_j^{(k)}|^2\Big)\\&\leq k^2\|f_k\|^2_{H^2}\,\Big(\sum_{j=1}^k\Big|\Big(\frac{P_{n-j}}{P_n}\Big)^{\alpha}-1\Big|^2\Big).
\end{align*}
By the Mean Value Theorem for every $j\in\{1,2,\cdots,k\}$ there is $\xi_j\in(P_{n-j},P_n)$ such that $|P_{n-j}^{\alpha}-P_n^{\alpha}|=\alpha\xi_j^{\alpha-1}\,|P_{n-j}-P_n|$.
The last sum can be estimated from above by 
\begin{align*}
\sum_{j=1}^k\Big|\Big(\frac{P_{n-j}}{P_n}\Big)^{\alpha}-1\Big|^2&\leq \frac{1}{P_n^{2\alpha}}\sum_{j=1}^{k}\alpha^2|P_{n-j}-P_n|^2\,P^{2\alpha-2}_n\leq \alpha^2\frac{p_n^2}{P^2_n}\sum_{j=1}^kj^2\leq  \alpha^2\frac{k^3}{n^2}\rho_n^2.
\end{align*}
Therefore for sufficiently large $n$ we get the estimate 
$$\sum_{j=1}^k\Big|\Big(\frac{P_{n-j}}{P_n}\Big)^{\alpha}-1\Big|^2\leq 2\alpha^2\rho^2\frac{k^3}{n^2},$$
consequently $C(n,\alpha,f_k)\leq 2\alpha^2\rho^2\,\|f_k\|^2_{H^2}k^5/n^2$. Convergence $\|f_k-f\|_{\mathcal D_{\omega}}$ implies $\mathcal D_{\zeta}(f_k)\to \mathcal D_{\zeta}(f)\to 0$ as $k\to+\infty$ for $\mu$-a.a. $\zeta\in\overline{\mathbb D}$.  
 For $\zeta\in\overline{\mathbb D}$, then by Lemma \ref{l:Aleman}, for every $k\in\mathbb N$ there is $g_k\in H^2$ such that $f_k(z)=a^{(k)}+(z-\zeta)\,g_k(z)$ for some $a^{(k)}\in\mathbb C$ and $\mathcal D_{\zeta}(f_k)=\|g_k\|^2_{H^2}$. Similarly for $f\in\mathcal D_{\zeta}$ there is $g\in H^2$ with $f(z)=a+(z-\zeta)\,g(z)$ for some $a\in\mathbb C$ and $\mathcal D_{\zeta}(f)=\|g\|^2_{H^2}$.
 Taking $\zeta\in\overline{\mathbb D}$ such that $\mathcal D_{\zeta}(f_k)\to \mathcal D_{\zeta}(f)$ implies $\|g_k\|_{H^2}\to\|g\|_{H^2}$. On the other hand if we write $g_k(z)=b_0^{(k)}+b_1^{(k)}z+\cdots+b_{k-1}^{(k)}z^{k-1}$ for $k\in\mathbb N_0$, then we have $$\|f_k\|^2_{H^2}\leq 2(|a^{(k)}|^2+|\zeta|^2\,\sum_{j=0}^{k-1}|b^{(k)}_j|^2+\sum_{j=0}^{k-1}|b_j^{(k)}|^2)=2(|a^{(k)}|^2+(1+|\zeta|^2)\,\|g_k\|^2_{H^2}).$$
 Moreover $\|f_k-f\|_{\mathcal D_{\omega}}\to 0$ yields in particular that $|f_k(0)-f(0)|=|a^{(k)}_0-a_0|\to 0$ as $k\to+\infty$, which together with $\sup_{k\in\mathbb N_0}|b^{(k)}_0|<+\infty$ (since $\|g_k\|_{H^2}\to\|g\|_{H^2}$ as $k\to+\infty$), give $\limsup_{k\to+\infty}|a^{(k)}|=|a^*|$ for a certain $a^*\in\mathbb C$. Consequently $\limsup_{k\to+\infty}\|f_k\|^2_{H^2}\leq 2(|a^*|^2+(1+|\zeta|^2)\|g\|^2_{H^2})\leq 2(|a^*|^2+2\|g\|^2_{H^2})$. Let $C(f)=2(|a^*|^2+2\|g\|^2_{H^2})=2(|a^*|^2+2\mathcal D_{\zeta}(f))$ then for sufficiently large $k$ (and hence for sufficiently large $n$) we have $C(n,\alpha,f_k)\leq 4\alpha^2\rho^2C(f)k^5/n^2$. Taking $n=k^5$ we can write $C(n,\alpha,f_k)\leq 4\alpha^2\rho^2C(f)/n$ and so
 $$\|N^{\alpha}_n[f]-f\|_{\mathcal D_{\omega}}\leq (2\alpha\rho^{\alpha}+2\alpha\rho (C(f)C(\omega))^{1/2}+1)\frac{1}{n^{1/2}},\quad\text{for all sufficiently large}\;n.$$
\end{proof}

We close this section with an analogue result when the sequence $(p_n)_{n\in\mathbb N_0}$ is non-increasing, provided that $\liminf_{n\to+\infty}p_n>0$. This last condition is essential in getting uniform upper bound for the operator norm $\|T^{\alpha}_n\|_{\ell^2\to\ell^2}$.
In this situation, unlike for non-decreasing sequences, the assumption of finite upper growth rate is redundant, since from non-increasing property it holds that $P_n\geq (n+1)p_n$ for all $n\in\mathbb N_0$, consequently $\rho_n\leq n/(n+1)$ and in particular $\limsup_{n\to+\infty}\rho_n\leq 1$. We state the next result without proof.

\begin{thm}
\label{th:non-increasing} Let $(N, (p_n)_{n\in\mathbb N_0})$ be the N\"orlund operator with $(p_n)_{n\in\mathbb N_0}$ non-increasing and satisfying $\liminf_{n\to+\infty}p_n>0$, and let $\alpha>1/2$. Then $\lim_{n\to+\infty}\|N^{\alpha}_n[f]-f\|_{\mathcal D_{\omega}}=0$ for every $f\in\mathcal D_{\omega}$. Moreover the convergence rate is of the order $O(n^{-1/2})$.
\end{thm}

\begin{question}
Is $\alpha=1/2$ a sharp bound in this case? Can the convergence rate in Theorem \ref{th:quantitative} and Theorem \ref{th:non-increasing} be improved?
\end{question}

\begin{question}
Is generalized N\"orlund method monotone? In other words does $N^{\alpha}_n[f]\to f$ imply $N^{\beta}_n[f]\to f$ when $\alpha\leq \beta$?
\end{question}

\subsection*{Discussion}
The sequence $(p_n)_{n\in\mathbb N_0}$ is assumed to be monotone. In fact this condition could be relaxed to eventually monotone, that is there is some $n_0\in\mathbb N_0$ such that $(p_n)_{n\geq n_0}$ is monotone. Indeed this would not affect the analysis since both the upper and lower estimates for the operator norm $\|T^{\alpha}_n\|^2_{\ell^2\to\ell^2}$ would only differ by an additive term (sum of $P_k$--like terms up to index $n_0$ divided by $p_n^{2\alpha}$) which in the limit would be bounded by a uniform positive constant. The situation for a general sequence $(p_n)_{n\in\mathbb N_0}$, not necessarily monotone, remains also to be investigated.
\bibliographystyle{plain}
\bibliography{references}

\begin{thebibliography}{10}

\bibitem{Aleman}
A.~Aleman.
\newblock {\em The {M}ultiplication {O}perator on {H}ilbert {S}paces of
  {A}nalytic {F}unctions}.
\newblock Habilitationsschrift, Fern Universit\"at, Hagen, 1993.

\bibitem{Boos}
J.~Boos.
\newblock {\em Classical and Modern Methods in Summability}.
\newblock Oxford University Press, Oxford, 2000.

\bibitem{Borwein}
D.~Borwein.
\newblock Norlund operators on $\ell_p$.
\newblock {\em Canad. Math. Bull.}, 36(1):8--14, 1993.

\bibitem{du-Bois}
P.~du~Bois-Raymond.
\newblock Zus\"atze zur {A}bhandlung: {U}ntersuchungen \"uber die {C}onvergenz
  und {D}ivergenz der {F}ourierschen {D}arstellungsformeln.
\newblock {\em Math. Ann.}, 10(1):431--445, 1876.

\bibitem{Fejer}
L.~Fej\'er.
\newblock Sur les fonctions int\'egrables et born\'ees.
\newblock {\em C.R. Acad. Sci. Paris}, pages 984--987, 1900.

\bibitem{Fejer2}
L.~Fej\'er.
\newblock Untersuchungen \"uber {F}ouriersche {R}eihen.
\newblock {\em Math. Ann.}, 58:51--69, 1904.

\bibitem{Mashregi-Ransford2}
P.~O.~Paris\'e J.~Mashreghi and Th. Ransford.
\newblock Ces\' aro summability of {T}aylor series in weighted {D}irichlet
  spaces.
\newblock {\em Complex Anal. Oper. Theory}, 15(7):1--8, 2021.

\bibitem{Mashregi-Ransford}
J.~Mashreghi and Th. Ransford.
\newblock Hadamard multipliers on weighted {D}irichlet spaces.
\newblock {\em Integr. Equ. Oper. Theory}, 91:1--13, 2019.

\bibitem{Ransford}
Th. Ransford.
\newblock {\em Potential {T}heory in the {C}omplex {P}lane}.
\newblock Cambridge University Press, New York, 1995.

\bibitem{Richter}
S.~Richter.
\newblock A representation theorem for cyclic analytic two-isometries.
\newblock {\em Trans. Amer. Math. Soc.}, 328(1):325--349, 1991.

\bibitem{Richter2}
S.~Richter and C.~Sundberg.
\newblock A formula for the local {D}irichlet integral.
\newblock {\em Michigan Math. J.}, 38(3):355--379, 1991.

\bibitem{Riesz}
M.~Riesz.
\newblock Sur les s\'eries de {D}irichlet et les s\'eries enti\`eres.
\newblock {\em C.R. Acad. Sci. Paris}, 149:309--312, 1911.

\bibitem{Spivak}
M.~Spivak.
\newblock {\em Calculus}.
\newblock Benjamin, New York, 1967.

\bibitem{Yosida}
K.~Yosida.
\newblock {\em Functional {A}nalysis}.
\newblock Springer-Verlag, Berlin, 4 edition, 1974.

\end{thebibliography}

\end{document}